\documentclass[12pt]{article}
\usepackage{authblk}
\usepackage{amsmath}
\usepackage{mathdots}
\usepackage{amsfonts}
\usepackage{amssymb}
\usepackage{latexsym}
\usepackage{graphicx}
\usepackage{pgfplots}
\usepackage{todonotes}

\usepackage{amsthm}

\usepackage{color}
\usepackage{fullpage}
\usepackage{verbatim}

\def\rhat{\hat r}
\newcommand\pthroot[1]{\sqrt[\leftroot{-2}\uproot{2}p]{#1}}

\DeclareMathOperator*{\argmin}{arg\,min}

\newtheorem{theorem}{Theorem}[section]
\newtheorem{lemma}{Lemma}[section]

\newcommand{\ignore}[1]{}

\title{
Approximating the pth root by composite rational functions
}
\author{Evan S. Gawlik\thanks{Department of Mathematics, University of Hawaii at Manoa, egawlik@hawaii.edu} \;and Yuji Nakatsukasa\thanks{Mathematical Institute, University of Oxford, and National Institute of Informatics, 
nakatsukasa@maths.ox.ac.uk}}

\date{}

\begin{document}
\maketitle

\begin{abstract}
A landmark result from rational approximation theory states that $x^{1/p}$ on $[0,1]$ can be approximated by a type-$(n,n)$ rational function with root-exponential accuracy. 
Motivated by the recursive optimality property of Zolotarev functions (for the square root and sign functions), we investigate approximating $x^{1/p}$ by composite rational functions of the form 
$r_k(x, r_{k-1}(x, r_{k-2}( \cdots (x,r_1(x,1)) )))$. 
While this class of rational functions ceases to contain the minimax (best) approximant for $p\geq 3$, we show that it achieves approximately $p$th-root exponential convergence with respect to the degree. 
Moreover, crucially, the convergence is \emph{doubly exponential} with respect to the number of degrees of freedom, suggesting that composite rational functions are able to approximate $x^{1/p}$ and related functions (such as $|x|$ and the sector function) with exceptional efficiency. 
\end{abstract}


\section{Introduction}

Composing rational functions is an efficient way of generating a rational function $r(x)=r_k(\cdots r_2(r_1(x)))$ of high degree: if each $r_i$ is of type $(m,m)$, then $r$ is of type $(m^k,m^k)$. By choosing each $r_i$ appropriately, one can often obtain a function $r$ that approximates a desired function in a wide domain of interest. 

There is no reason to expect---and it is generally not true---that 
$r_k(\cdots r_2(r_1(x)))$ can express the minimax rational approximant of a given type, say $(m^k,m^k)$, to a given function. 
However, building upon Rutishauser~\cite{rutishauser1963betrachtungen} and Ninomiya~\cite{ninomiya1970best}, Nakatsukasa and Freund~\cite{mysirev} show a remarkable property of the best rational approximants to the function $\operatorname{sign}(x)=x/|x|$ on $[-1,-\delta]\cup [\delta,1]$ for $0<\delta<1$ (called Zolotarev functions): appropriately composing Zolotarev functions gives another Zolotarev function of higher degree. In other words, the class of composite rational functions $r(x)=r_k(\cdots r_2(r_1(x)))$, with each $r_i$ of type $(m,m)$, contains the type-$(m^k,m^k)$ minimax approximant to the sign function.
Moreover, for a fixed $\delta$, the convergence of Zolotarev functions is exponential in the degree. Since the degree is $m^k$, and the number of parameters necessary to express $r$ is $d\approx 2km$, 
it follows that the convergence is $\exp(-m^k) = \exp(-\exp(C d))$, a \emph{double-exponential} convergence rate. This is so powerful that choosing $m=17$ and $k=2$ (one composition, i.e., two iterations) is enough to obtain convergence to machine precision in double precision arithmetic, with error below $10^{-15}$. 

Functions related to the sign function, such as 
$|x|$ (via $|x|=x/\mbox{sign}(x)$) and $\sqrt{x}$ 
(via $|x|\approx p(x^2)/q(x^2)$ then $\sqrt{x}\approx p(x)/q(x)$) can similarly be approximated by composite rational functions. Gawlik~\cite{gawlik2018zolotarev} does this for the square root and shows that 
a composite rational function yields the minimax rational approximant (in the relative sense) on intervals $[\delta,1] \subset (0,1]$, and that the approximation extends far into the complex plane.  This observation generalizes earlier work on rational approximation of the square root with optimally scaled Newton iterations~\cite{Beckermannmatfun,ninomiya1970best,rutishauser1963betrachtungen,wachspress1962}.
Moreover, an extension was derived in \cite{gawlik2018pth}, which shows that the $p$th root can be approximated efficiently on intervals $[\delta,1] \subset (0,1]$, although not with minimax quality. 

Clearly, in the above papers the origin is excluded from the domain, as the functions have a singularity at $x=0$. However, a landmark result from rational approximation theory~\cite{gonvcar1967rapidity,stahl2003best} states that the best rational approximant (in the absolute sense) of $x^\beta$ (for any real $\beta >0$) on $[0,1]$ can be approximated by a type-$(n,n)$ rational function with root-exponential accuracy. One might wonder, can this be done with a composite rational function? This is the question we address in this paper.  We focus on the case in which $\beta = 1/p$ with $p \ge 2$ an integer.

We show that a rational function of the form $r(x) = r_k(x, r_{k-1}(x, r_{k-2}( \cdots (x,r_1(x,1)) )))$ can approximate $x^{1/p}$ on $[0,1]$ with superalgebraic accuracy, with close to $p$th root-exponential convergence. Moreover---and crucially---the convergence is \emph{doubly exponential} with respect to the number of degrees of freedom.  That is, the error is $O(\exp(-c_1\exp(c_2d)))$ for some constants $c_1,c_2>0$, where $d$ is the number of parameters needed to express the rational function.  By ``number of parameters'' we mean $d=\sum_{i=1}^k m_i+\ell_i+1$ if $r_i$ has type $(m_i,\ell_i)$ for $i=1,2,\dots,k$, so that $d$ reflects the cost of evaluating $r$ at a matrix argument. 

Clearly, our result implies that any rational power of $x$ can be 
approximated by a composite rational function. Moreover, 
since $|r(x)-x^{1/p}|\leq \epsilon$ on $[0,1]$ implies 
$|r(x/s)- (x/s)^{1/p}|\leq \epsilon$ on $[0,s]$ for any $s>0$, hence 
$|s^{1/p}r(x/s)- x^{1/p}|\leq s^{1/p}\epsilon$, 
our results also show that any rational power can be approximated efficiently on $[0,s]$ by a composite rational function. In addition, our approximants to $x^{1/p}$ immediately lead to approximants to the $p$-sector function $\operatorname{sect}_p(z)=z/(z^p)^{1/p}$.

More generally, we think composite (rational) functions are a powerful tool in approximation theory, and we regard this as a contribution towards demonstrating their effectiveness and practicality. 
Indeed, one might say they are already used extensively in scientific computing: 
\begin{enumerate}
\item Composite rational functions are implicitly employed in most algorithms for computing matrix functions~\cite{Higham2008FM}, in which approximating a function on the spectrum of the matrix is required. For the $p$th root, a standard algorithm \cite[Ch.~7]{Higham2008FM} employs Newton's method, which ultimately approximates $A^{1/p}$ with a sequence of rational functions $f_k$ of $A$ given recursively by $f_{k+1}(x) = \frac{1}{p}((p-1)f_k(x)+x/f_k(x)^{p-1})$, $f_0(x)=1$.   The function $f_k$ is composite rational and similar to the approximants we use, but not the same (it is unscaled), and it exhibits exponential rather than double-exponential convergence on $[0,1]$.  
Generally speaking, Newton's method for computing a matrix function $f(A)$ (or more generally for various nonlinear problems, e.g. rootfinding) can often be interpreted as approximating $f(A)$ (or the solution) by a composite rational function of $A$.
\item The rapidly growing subject of deep learning is based on composing a large number of nonlinear activation functions~\cite{lecun2015deep}. 
\end{enumerate}

\paragraph{Summary of Results.}
To summarize our results, let us introduce some terminology.
We say that a univariate rational function $r(x)=p(x)/q(x)$ is of type $(m,\ell)$ if $p$ and $q$ are polynomials of degrees at most $m$ and $\ell$, respectively.
We denote the set of all such rational functions by $\mathcal{R}_{m,\ell}$.
We say that a bivariate rational function $r(x,y)$ is of type $(m,\ell)$ if $r(x,x)$ is of type $(m,\ell)$.  
We say that a univariate rational function $r$ is $(k,m,\ell)$-composite if $r$ is a composition of $k$ rational functions $r_i(x,y)$, $i=1,2,\dots,k$, each of type $(m,\ell)$:
\begin{equation} \label{comp}
r(x) = r_k(x, r_{k-1}(x, r_{k-2}( \cdots (x,r_1(x,1))))).
\end{equation}

Here is the main result of this paper. 
\begin{theorem}\label{thm:main}
Let $p \ge 2$ be an integer.  There exists a positive constant $N$ depending on $p$ such that for every integer $n \ge N$, there exists a $(\lfloor \log_p n \rfloor +1,p,p-1)$-composite rational function $r$ of type $(n,n-1)$ such that
\begin{equation}  \label{eq:mainthm}
\max_{x \in [0,1]} |r(x)-x^{1/p}| \le \exp( -b n^c ),
\end{equation}
where $b>0$ is a constant depending on $p$ and
\begin{equation}
  \label{eq:exponentc}
c = \frac{ \log \left(\frac{p}{p-1}\right) \log 2 }{ \log \left(\frac{2p}{p-1}\right) \log p }.
\end{equation}
\end{theorem}
Note that when $p=2$, $c=\frac{1}{2}$, and as $p \rightarrow \infty$, $c \sim \frac{1}{p\log p}$.

Let us comment on the theorem. The bound~\eqref{eq:mainthm} shows that by using a $(\lfloor \log_p n \rfloor +1,p,p-1)$-composite rational function we can approximate the $p$th root with ``$1/c$th root''-(nearly $p$th root) exponential accuracy with respect to the degree, which is suboptimal unless $p=2$ (in which case a composite rational function on $[\delta,1]$ is optimal in the relative sense). 

However, the result is still striking in the following sense: the number of degrees of freedom used to express $r$ is just $O(pk)$ for $n\approx p^k$ (see below \eqref{alphaupdate}), and therefore with respect to the degrees of freedom $d$, 
the convergence is 
\begin{equation}\label{eq:DEconvergence}
\max_{x \in [0,1]} |r(x)-x^{1/p}| \le \exp( -b p^{\tilde cd} )  ,
\end{equation} indicating a \emph{double-exponential} convergence with respect to $d$. 

As a byproduct of our analysis, we will obtain analogous results for composite rational approximation of the $p$-sector function $\operatorname{sect}_p(z)=z/(z^p)^{1/p}$ on the set $S_p \subset \mathbb{C}$ given by
\begin{equation} \label{Sp}
S_p = \{ x e^{2\pi i j/p}  \mid x \in [0,1], \, j \in \{1,2,\dots,p\} \}.
\end{equation}
We will also consider the subset $S_{p,\alpha}$ of $S_p$ excluding the origin 
\begin{equation} \label{Sptil}
S_{p,\alpha} = \{ x e^{2\pi i j/p}  \mid x \in [\alpha,1], \, j \in \{1,2,\dots,p\} \}.
\end{equation}

We say that a $(k,m,\ell)$-composite rational function~(\ref{comp}) is \emph{pure} if the functions $r_j(x,y)$ appearing in~(\ref{comp}) are univariate:
\[
r(x) = r_k(r_{k-1}( \cdots (r_1(x)))).
\]
\begin{theorem} \label{thm:sector}
Let $p \ge 2$ be an integer, and $\alpha \in (0,1)$.  
There exists a positive constant $N$ depending on $p$ such that for every integer $n \ge N$, there exist pure $(\lfloor \log_p n \rfloor,1,p)$-composite rational functions $r$ and $q$ of type $(n-p+1,n)$ such that
\begin{equation}  \label{eq:sectthm}
\max_{z \in S_p} \left|z \left( r(z) - \operatorname{sect}_p(z) \right)\right| \le \exp( -b n^c ),
\end{equation}
where $b$ and $c$ are as in Theorem~\ref{thm:main}, and 
\begin{equation}  \label{eq:sectthmhat}
\max_{z \in S_{p,\alpha}} \left| q(z) - \operatorname{sect}_p(z) \right| \le \exp( -{\widehat  b} n^{\widehat c} ),
\end{equation}
where $\widehat b>0$ depends on $\alpha$ and $p$, and $\widehat c = \frac{\log 2}{\log p}$.  
\end{theorem}
It is worth noting that the two rational functions $r$, $q$ are 
generally different---they coincide for a particular value of $\alpha$.
The error in~(\ref{eq:sectthm}) is measured in a weighted norm, which is natural in view of the fact that $\operatorname{sect}_p(z)$ is discontinuous at $z=0$. When $p=2$ and $z \in S_2$, $z\operatorname{sect}_p(z)=|z|$ and 
$c=\frac{1}{2}$, so \eqref{eq:sectthm} recovers the root-exponential convergence of rational approximants to $|x|$ on $[-1,1]$~\cite[Ch.~25]{trefethenatap}. 
By contrast,~\eqref{eq:sectthmhat} shows that a better bound holds for the absolute error if one excludes the neighborhood of the origin. When $p=2$, $\widehat c=1$ and \eqref{eq:sectthmhat} recovers the exponential convergence of Zolotarev functions to the sign function on $[-1,-\alpha]\cup [\alpha,1]$~\cite{elementselliptic,beckermann2017singular}. Our analysis will show that $\widehat b$ decays like a negative power of $\log\frac{1}{\alpha}$ as $\alpha \rightarrow 0$.

\paragraph{Organization.}  
This paper is organized as follows.  In Section~\ref{sec:composite}, we review some theory from~\cite{gawlik2018pth} concerning composite rational approximants of the $p$th root on positive real intervals.  In Section~\ref{sec:bounding}, we study the behavior of these approximants near the origin.  We then prove Theorems~\ref{thm:main} and~\ref{thm:sector} in Section~\ref{sec:convergence}, and we illustrate our results numerically in Section~\ref{sec:examples}.

\section{Composite rational approximation of the $p$th root} \label{sec:composite}
To approximate $x^{1/p}$ on an interval $[\alpha^p,1] \subset (0,1]$, Gawlik~\cite{gawlik2018pth} considers the recursively defined rational function
\begin{align}
f_{k+1}(x) &= f_k(x) \rhat_{m,\ell}\left(\frac{x}{f_k(x)^p},\alpha_k,\pthroot{\;\cdot\;}\right), &\quad f_0(x)&=1, \label{eq:recurse1} \\
\alpha_{k+1} &= \frac{ \alpha_k }{ \rhat_{m,\ell}\left(\alpha_k^p,\alpha_k,\pthroot{\;\cdot\;}\right) }, &\quad \alpha_0 &= \alpha, \label{eq:recurse2}
\end{align}
where $\rhat_{m,\ell}(x,\alpha,\pthroot{\;\cdot\;})$ is (a rescaling of) the \emph{relative} minimax rational approximant of type $(m,\ell) \in \mathbb{N}_0 \times \mathbb{N}_0 \setminus \{(0,0)\}$ on the interval $[\alpha^p,1]$:
\[
\hat{r}_{m,\ell}(x,\alpha,\pthroot{\;\cdot\;}) = \left( \frac{1+\alpha}{2\alpha} \right) r_{m,\ell}(x,\alpha,\pthroot{\;\cdot\;}),
\]
where
\begin{equation} \label{rml}
r_{m,\ell}(\;\cdot\;, \alpha, \pthroot{\;\cdot\;}) = \argmin_{r \in \mathcal{R}_{m,\ell}} \max_{x \in [\alpha^p,1]} \left| \frac{r(x)-x^{1/p}}{x^{1/p}} \right|.
\end{equation}
Gawlik shows that $f_k(x)$ is a rapidly convergent approximant to the $p$th root on $[\alpha^p,1]$. With $k$ recursions, the maximum relative error $|f_k(x)-x^{1/p}|/|x^{1/p}|$ on $[\alpha^p,1]$ decays double exponentially in $k$: it is bounded above by $c_1 \exp(-c_2(m+\ell+1)^k)$ for some $c_1,c_2>0$ depending on $m$, $\ell$, $p$, and $\alpha$.  Importantly, these constants depend very weakly on $\alpha$; the analysis below will implicitly show that when $(m,\ell)=(1,0)$, $c_1$ is independent of $\alpha$ and $c_2$ decays like a negative power of $\log\frac{1}{\alpha}$ as $\alpha \rightarrow 0$, just like $\widehat b$ in~\eqref{eq:sectthmhat}.

Given that~\eqref{eq:recurse1} is an approximant on $[\alpha^p,1]$, which is an interval that excludes the singularity at $x=0$, 
a natural question arises: can we approximate on $[0,1]$? 
Intuitively, 
the function is still continuous at $x=0$ (unlike e.g. the sign or sector function) with $0^{1/p}=0$, and hence it is possible to approximate $x^{1/p}$ on the whole interval $[0,1]$. Indeed Stahl~\cite{stahl2003best} shows that $x^{1/p}$ on $[0,1]$ can be approximated by a type-$(n,n)$ rational function with root-exponential accuracy (we refer to~\cite{braess1986nonlinear,petrushev2011rational} for general results on classical rational approximation theory). 
Can a highly efficient rational approximant be constructed based on recursion as in~\eqref{eq:recurse1}? 
It is important to note that we will necessarily switch to the (more natural) metric of absolute error $|r(x)-x^{1/p}|$ rather than the relative error $|r(x)-x^{1/p}|/|x^{1/p}|$ for this purpose. 

It turns out that the rational function \eqref{eq:recurse1} does a good job approximating on $[0,1]$, when $\alpha$ is chosen carefully: when it is too small, the error is large on $[\alpha^p,1]$ (in fact it is maximal at $x=1$~\cite{gawlik2018pth}).
Conversely if $\alpha$ is too large, the error is large on $[0,\alpha^p]$ (in fact it is $O(\alpha)$ at $x=0$, as we show below). 
A major task undertaken in what follows is to choose  $\alpha$ so that the convergence is optimized, in that the error on $[0,\alpha^p]$ and $[\alpha^p,1]$ are balanced to be approximately the same. 

Our analysis will focus on the lowest-order version of the iteration~(\ref{eq:recurse1}-\ref{eq:recurse2}),  obtained by choosing $(m,\ell)=(1,0)$.  It is shown in~\cite[Proposition 5]{gawlik2018pth} (and elsewhere~\cite{king1971improved,meinardus1980optimal}) that for this choice of $m$ and $\ell$, 
\begin{equation}
  \label{eq:m1ell0}
\rhat_{1,0}(x,\alpha,\pthroot{\;\cdot\;})=\frac{1}{p}
\left((p-1)\mu(\alpha)+\frac{x}{\mu(\alpha)^{p-1}}
\right),\qquad 
\mu(\alpha) =\left(\frac{\alpha-\alpha^p}{(p-1)(1-\alpha)}\right)^{1/p}.
\end{equation}
Thus, when $(m,\ell)=(1,0)$, the iteration~(\ref{eq:recurse1}-\ref{eq:recurse2}) reads
\begin{align}
f_{k+1}(x) &= \frac{1}{p} \left( (p-1) \mu(\alpha_k) f_k(x) + \frac{x}{\mu(\alpha_k)^{p-1} f_k(x)^{p-1}} \right), &\quad f_0(x) &= 1, \label{fupdate} \\
\alpha_{k+1} &= \frac{p \alpha_k}{(p-1)\mu(\alpha_k) + \mu(\alpha_k)^{1-p} \alpha_k^p} , &\quad \alpha_0 &= \alpha. \label{alphaupdate}
\end{align}
Note that $f_k$ is $(k,p,p-1)$-composite since it is of the form~(\ref{comp}) with
\[
r_j(x,y) = \frac{1}{p}\left( \frac{ (p-1) \mu(\alpha_{j-1})^p y^p + x }{ \mu(\alpha_{j-1})^{p-1} y^{p-1} } \right)
\]
for each $j$.  It follows from this observation and an inductive argument that $f_k$ has type $(p^{k-1},p^{k-1}-1)$ for each $k \ge 1$.

We rely heavily on this explicit expression for the particular case $(m,\ell)=(1,0)$, as it lets us analyze the functions in detail, which leads to a constructive proof for Theorem~\ref{thm:main}. 
We note that using larger values of $(m,\ell)$ may result in faster convergence, in particular a larger exponent $c$ than~\eqref{eq:exponentc}. In view of~\eqref{eq:DEconvergence}, the convergence is still doubly exponential, with an improved constant $\tilde c$. However, we do not expect the improvement would be significant.

Moreover, composing low-degree rational functions is an extremely efficient way to construct high-degree rational functions of matrices, and we suspect that our choice $(m,\ell)=(1,0)$ would give the fastest convergence in terms of the number of matrix operations needed to evaluate $r$ at a matrix argument.

\section{Bounding the error on \boldmath${[0,\alpha^p]}$} \label{sec:bounding}

In this section, we analyze the absolute error committed by the function $f_k$ defined by~\eqref{fupdate}--\eqref{alphaupdate} on the interval $[0,\alpha^p]$.  It will be convenient to consider not $f_k$ but the scaled function 
\begin{equation} \label{fktilde}
\widetilde{f}_k(x) = \frac{2\alpha_k}{1+\alpha_k} f_k(x),
\end{equation}
which has the property that~\cite[Theorem 2]{gawlik2018pth}
\begin{equation} \label{eq:error1}
\max_{x \in [\alpha^p,1]} \frac{ \widetilde{f}_k(x) - x^{1/p} } {x^{1/p}} = -\min_{x \in [\alpha^p,1]} \frac{ \widetilde{f}_k(x) - x^{1/p} } {x^{1/p}}  = \frac{1-\alpha_k}{1+\alpha_k} \in (0,1).
\end{equation}
We will prove the following estimate.

\begin{theorem} \label{thm:error0}
Let $\alpha \in (0,1)$.  The function $\widetilde{f}_k$ defined by~\eqref{fupdate}--\eqref{alphaupdate} and~\eqref{fktilde} satisfies
\begin{equation}  \label{eq:bound0alp}
\max_{x \in [0,\alpha^p]} | \widetilde{f}_k(x) - x^{1/p}  | \le 2\alpha  
\end{equation}
for every $k \ge 0$.
\end{theorem}

Experiments suggest that the bound~\eqref{eq:bound0alp} could be improved to $<\alpha$ for $k$ large enough, but this does not affect what follows in any significant way.

We will prove Theorem~\ref{thm:error0} by a series of lemmas.  Let
\[
g_k(x) = \frac{x}{f_k(x^p)}.
\]
Note that $g_0(x) = x$ and
\begin{align}\label{eq:sectorg}
g_{k+1}(x) = \frac{x}{ f_k(x^p) \hat{r}_{1,0}\left( \frac{x^p}{f_k(x^p)^p}, \alpha_k, \pthroot{\;\cdot\;} \right) } = \frac{g_k(x)}{ \hat{r}_{1,0}(g_k(x)^p, \alpha_k,\pthroot{\;\cdot\;}) } = \hat{s}(g_k(x), \alpha_k),
\end{align}
where
\[
\hat{s}(x,\alpha) = \frac{x}{\hat{r}_{1,0}(x^p,\alpha,\pthroot{\;\cdot\;})} = \frac{px}{(p-1)\mu(\alpha) + \mu(\alpha)^{1-p} x^p}.
\]
Also let
\[
H(\alpha) = \hat{s}(\alpha,\alpha) = \frac{p \alpha}{(p-1)\mu(\alpha) + \mu(\alpha)^{1-p} \alpha^p},
\]
so that $\alpha_{k+1} = H(\alpha_k)$.

\begin{lemma} \label{lemma:shat}
For every $\alpha \in (0,1)$ and every $x \in [0,\alpha]$,
\[
0 \le x \hat{s}'(x,\alpha) \le \hat{s}(x,\alpha) \le H(\alpha),
\]
where the prime denotes differentiation with respect to $x$.
\end{lemma}
\begin{proof}
A short calculation shows that
\[
x \hat{s}'(x,\alpha) = w(x) \hat{s}(x,\alpha),
\]
where 
\[
w(x) = \frac{ (p-1)\left(1 - \left(\frac{x}{\mu(\alpha)}\right)^p \right) }{ (p-1) + \left(\frac{x}{\mu(\alpha)}\right)^p  }.
\]
Since $0 \le w(x) \le 1$ for every $x \in [0,\mu(\alpha)]$, it follows that 
\[
0 \le x \hat{s}'(x,\alpha) \le \hat{s}(x,\alpha), \quad x \in [0,\mu(\alpha)].
\]
In particular, the above inequalities hold on $[0,\alpha] \subset [0,\mu(\alpha)]$, and $\hat{s}(x,\alpha)$ is nondecreasing on $[0,\alpha]$.  Thus,
\[
\hat{s}(x,\alpha) \le \hat{s}(\alpha,\alpha) = H(\alpha), \quad x \in [0,\alpha].
\]
\end{proof}

Now let $\alpha \in (0,1)$ be fixed.
\begin{lemma} \label{lemma:gk}
For every $x \in [0,\alpha]$ and every $k \ge 0$,
\[
0 \le x g_k'(x) \le g_k(x) \le \alpha_k.
\]
\end{lemma}
\begin{proof}
Since $g_0(x)=x$ and $\alpha_0=\alpha$, the above inequalities hold when $k=0$.  Assume that they hold for some $k \ge 0$.
Observe that
\[
x g_{k+1}'(x) = x g_k'(x) \hat{s}'(g_k(x), \alpha_k) .
\]
Since $g_k(x) \in [0,\alpha_k]$ for $x \in [0,\alpha]$, Lemma~\ref{lemma:shat} implies that $\hat{s}'(g_k(x), \alpha_k) \ge 0$.  It follows from this and our inductive hypothesis that $xg_{k+1}'(x) \ge 0$ for $x \in [0,\alpha]$.  In addition, since $x g_k'(x) \le g_k(x)$ and $g_k(x) \hat{s}'(g_k(x),\alpha_k) \le \hat{s}(g_k(x),\alpha_k)$,
\[
x g_{k+1}'(x) \le \hat{s}(g_k(x),\alpha_k) = g_{k+1}(x).
\]
Finally, since $\hat{s}(g_k(x),\alpha_k) \le H(\alpha_k) = \alpha_{k+1}$, it follows that $g_{k+1}(x) \le \alpha_{k+1}$.
\end{proof}

\begin{lemma}
For every $x \in [0,\alpha^p]$ and every $k \ge 0$,
\[
0 < \widetilde{f}_k(x) \le \alpha(1+\varepsilon_k), \quad \varepsilon_k = \frac{1-\alpha_k}{1+\alpha_k}.
\]
\end{lemma}
\begin{proof}
We first note that $f_k$ is positive and nondecreasing on $[0,\alpha^p]$.  Indeed, differentiating the relation
\[
f_k(x^p) = \frac{x}{g_k(x)}
\]
gives
\[
p x^{p-1} f_k'(x^p) = \frac{g_k(x) - x g_k'(x)}{g_k(x)^2},
\]
so Lemma~\ref{lemma:gk} implies that $f_k'(x^p) \ge 0$ for every $x \in [0,\alpha]$.  
Evaluating the recursion~(\ref{fupdate}) at $x=0$ gives
\[
f_{k+1}(0) = f_k(0) \left( \frac{p-1}{p} \right) \mu(\alpha_k), \quad f_0(0) = 1,
\]
so $f_k(0)>0$ for every $k$.
Since $\widetilde{f}_k(x)$ is a positive multiple of $f_k(x)$, it follows that $0 < \widetilde{f}_k(x) \le\widetilde{f}_k(\alpha^p)$ for every $x \in [0,\alpha^p]$.  Finally, taking $x=\alpha^p$ in~(\ref{eq:error1}) gives $\widetilde{f}_k(\alpha^p) \le \alpha(1+\varepsilon_k)$.
\end{proof}

By the lemma above,
\[
|\widetilde{f}_k(x) - x^{1/p}| \le \max\{ |\widetilde{f}_k(x)|, |x^{1/p}| \} \le \max\{ \alpha(1+\varepsilon_k),  \alpha \} = \alpha(1+\varepsilon_k) \le 2 \alpha, \quad x \in [0,\alpha^p],
\]
so
\[
\max_{x \in [0,\alpha^p]} |\widetilde{f}_k(x) - x^{1/p}| \le 2\alpha.
\]
This completes the proof of Theorem~\ref{thm:error0}.

An estimate for the absolute error on $[0,1]$ is now immediate: Combining the above theorem,~(\ref{eq:error1}), and the fact that $x^{1/p} \le 1$ for $x \in [0,1]$, we see that
\begin{equation} \label{combinedbound}
\max_{x \in [0,1]} | \widetilde{f}_k(x) - x^{1/p}  | \le \max\left\{ 2\alpha, \frac{1-\alpha_k}{1+\alpha_k} \right\}.
\end{equation}

\subsection{Sector function approximation}
We note that the function $g_k$ in~\eqref{eq:sectorg}
approximates the $p$-sector function $\mbox{sect}_p(z)=z/(z^p)^{1/p}$ (this observation appeared in~\cite[Sec.~4]{gawlik2018pth}), and 
$g_k$ is a \emph{pure} composite rational function of the form 
$g_k(z) = r_k(r_{k-1}(\cdots r_2(r_1(z))))$. 
In fact it is $(k,1,p)$-composite, and an inductive argument shows that it has type $(p^k-p+1,p^k)$.
In the $p=2$ case, this reduces to Zolotarev's best rational 
approximant to the sign function of type $(2^k-1,2^k)$.  That is, as in the square root approximation, the minimax rational approximant is contained in the class of (here purely) composite rational functions. 

Below we derive estimates for the maximum weighted error $|z(g_k(z) - \operatorname{sect}_p(z))|$ on the sets $S_p,S_{p,\alpha} \subset \mathbb{C}$ defined in~(\ref{Sp}) and~\eqref{Sptil}.  As before, it will be convenient to work not with $g_k(z)$ but with the rescaled function
\[
\widetilde{g}_k(z) = \frac{2}{1+\alpha_k} g_k(z) =  \frac{4\alpha_k}{(1+\alpha_k)^2} \frac{z}{\widetilde{f}_k(z^p)}.
\]
As shown in \cite[Sec.~4]{gawlik2018pth},  the relative error $\frac{\widetilde{g}_k(z)-\operatorname{sect}_p(z)}{\operatorname{sect}_p(z)}$ is real-valued and equioscillates on each line segment $\{z\in\mathbb{C} \mid e^{-2\pi ij/p}z\in[\alpha,1] \}$, $j=0,1,\ldots,p-1$. 
Note that here the relative and absolute errors are the same in modulus. The asymptotic convergence rate on $S_{p,\alpha}$ was analyzed in~\cite{gawlik2018pth}. Here we quantify the non-asymptotic convergence on $S_p$.
\begin{lemma}
For every $k \ge 0$,
\begin{equation} \label{eq:error1sector}
\max_{z \in S_p} |z (\widetilde{g}_k(z) - \operatorname{sect}_p(z))| \le \max\left\{ \alpha, \frac{1-\alpha_k}{1+\alpha_k} \right\},
\end{equation}
and 
\begin{equation} \label{eq:error1sectorhat}
\max_{z \in S_{p,\alpha}} |\widetilde{g}_k(z) - \operatorname{sect}_p(z)| \le \frac{1-\alpha_k}{1+\alpha_k} .
\end{equation}
\end{lemma}
\begin{proof}
Let $z = x^{1/p} e^{2\pi i j/p}$ with $x \in [0,1]$ and $j \in \{1,\dots,p\}$.  Since $\widetilde{g}_k(z) = e^{2\pi i j/p} \widetilde{g}_k(x^{1/p})$ and $\operatorname{sect}_p(z) = e^{2\pi i j/p}$, we have
\[
|z (\widetilde{g}_k(z) - \operatorname{sect}_p(z))| = |x^{1/p} (\widetilde{g}_k(x^{1/p}) - 1)|.
\]
If $x \in [0,\alpha^p]$, then Lemma~\ref{lemma:gk} implies that $0 \le \widetilde{g}_k(x^{1/p}) \le \frac{2\alpha_k}{1+\alpha_k} < 1$, so
\[
|x^{1/p} (\widetilde{g}_k(x^{1/p}) - 1)| \le x^{1/p} \le \alpha, \quad x \in [0,\alpha^p].
\]
On the other hand, if $x \in [\alpha^p,1]$, then
\begin{equation}  \label{eq:removexscale}
|x^{1/p} (\widetilde{g}_k(x^{1/p}) - 1)| \le |\widetilde{g}_k(x^{1/p}) - 1| = \left| \frac{4\alpha_k}{(1+\alpha_k)^2} \frac{x^{1/p}}{\widetilde{f}_k(x)} - 1\right|.  
\end{equation}
By~(\ref{eq:error1}),
\[
\frac{\widetilde{f}_k(x)}{x^{1/p}} \in \left[ 1 - \left( \frac{1-\alpha_k}{1+\alpha_k} \right), 1 + \left( \frac{1-\alpha_k}{1+\alpha_k} \right) \right] = \left[ \frac{2\alpha_k}{1+\alpha_k}, \frac{2}{1+\alpha_k} \right] , \quad x \in [\alpha^p,1],
\]
so
\[
\frac{x^{1/p}}{\widetilde{f}_k(x)} \in \left[ \frac{1+\alpha_k}{2}, \frac{1+\alpha_k}{2\alpha_k} \right] , \quad x \in [\alpha^p,1],
\]
and hence
\begin{equation}
  \label{eq:tildeuse}
\frac{4\alpha_k}{(1+\alpha_k)^2} \frac{x^{1/p}}{\widetilde{f}_k(x)} - 1 \in \left[ -\frac{1-\alpha_k}{1+\alpha_k}, \frac{1-\alpha_k}{1+\alpha_k} \right] , \quad x \in [\alpha^p,1].  
\end{equation}
It follows that 
\[
|x^{1/p} (\widetilde{g}_k(x^{1/p}) - 1)| \le \frac{1-\alpha_k}{1+\alpha_k}, \quad x \in [\alpha^p,1].
\]
For~\eqref{eq:error1sectorhat}, we simply start from the second expression in~\eqref{eq:removexscale} and use~\eqref{eq:tildeuse}.
\end{proof}

\section{Proof of Theorems~\ref{thm:main} and~\ref{thm:sector}} \label{sec:convergence}
To examine the convergence of the recursion~\eqref{fupdate}-\eqref{alphaupdate} on $[0,1]$, we first ask the question: given $\epsilon>0$, what values of $k$ and $\alpha$ are needed to get an error $\epsilon$?  In view of~(\ref{combinedbound}), we must choose $\alpha \le \epsilon/2$ and $k$ large enough so that $\frac{1-\alpha_k}{1+\alpha_k} \le \epsilon$.

To determine $k$, we select a constant $\alpha^* \in (1/e,1)$ (depending on $p$) and split the convergence of $\alpha_k\rightarrow 1$ into three stages: 
\begin{enumerate}
\item Find $k_1$ such that $\alpha_{k_1} \ge \frac{1}{e}$. 
\item Find $k_2$ such that $\alpha_{k_1+k_2} \ge \alpha^*$. 
\item Find $k_3$ such that $\frac{1-\alpha_{k_1+k_2+k_3}}{1+\alpha_{k_1+k_2+k_3}} \le \epsilon$. 
\end{enumerate}
Clearly, the second stage is independent of $\epsilon$ and $\alpha_0$, so $k_2$ is a constant (depending on $p$). 

Our choice of $\alpha^*$ is described in the following lemma.

\begin{lemma} \label{lemma:alphastar}
There exists a constant $\alpha^* \in (0,1)$ depending on $p$ such that  
\[
\frac{1-H(\alpha)}{1+H(\alpha)} \le \frac{p}{2}  \left(\frac{1-\alpha}{1+\alpha}\right)^2
\]
for every $\alpha \in [\alpha^*,1]$.
\end{lemma}
\begin{proof}
It is proven in~\cite[Theorem 2]{gawlik2018pth} that the iteration~(\ref{eq:recurse2}) generates an increasing sequence $\{\alpha_k\}_{k=0}^\infty$ satisfying $\lim_{k \rightarrow \infty} \alpha_k = 1$ and
\[
\frac{1-\alpha_{k+1}}{1+\alpha_{k+1}} = C(m,\ell,p) \left( \frac{1-\alpha_k}{1+\alpha_k} \right)^{m+\ell+1} + o \left( \left( \frac{1-\alpha_k}{1+\alpha_k} \right)^{m+\ell+1} \right),
\]
where
\[
C(m,\ell,p) = \frac{p^{m+\ell+1} m! \ell! (1/p)_{\ell+1} (1-1/p)_m }{ 2^{m+\ell} (m+\ell+1)!(m+\ell)! }
\]
and $(\beta)_m$ denotes the rising factorial (the Pochhammer symbol): $(\beta)_m = \beta(\beta+1)(\beta+2)\cdots(\beta+m-1)$.  Since $C(1,0,p) = \frac{p-1}{4}$, this implies that the iteration~(\ref{eq:recurse2}) with $(m,\ell)=(1,0)$ (i.e., the iteration~(\ref{alphaupdate})) generates $\{\alpha_k\}_{k=0}^\infty$ satisfying
\[
\frac{1-\alpha_{k+1}}{1+\alpha_{k+1}} = \left( \frac{p-1}{4} \right) \left( \frac{1-\alpha_k}{1+\alpha_k} \right)^2 + o \left( \left( \frac{1-\alpha_k}{1+\alpha_k} \right)^2 \right).
\]
In other words,
\[
\left. \frac{1-H(\alpha)}{1+H(\alpha)} \middle/ \left(\frac{1-\alpha}{1+\alpha}\right)^2 \right.  \rightarrow \frac{p-1}{4}, \text{ as } \alpha \uparrow 1.
\]
It follows that the above ratio is bounded by $\frac{p}{2}$ for $\alpha$ close enough to $1$.
\end{proof}
Without loss of generality, we assume 
\begin{equation} \label{eq:wlog}
\alpha^*  > \max\left\{\frac{1}{e},\frac{p-2}{p+2}\right\}
\end{equation}
in what follows.

\paragraph{Stage 1}
We will now determine $k_1$ such that $\alpha_{k_1} \ge \frac{1}{e}$.  We begin with a lemma.

\begin{lemma} \label{lemma:lowerbound}
For every $\alpha \in (0,1)$, 
\[
H(\alpha) > \alpha^{1-1/p}.
\]
\end{lemma}
\begin{proof}
We have
\begin{align*}
H(\alpha) 
&= \frac{p \alpha \mu(\alpha)^{p-1} }{ (p-1)\mu(\alpha)^p + \alpha^p } 
= \frac{p \alpha \mu(\alpha)^{p-1} }{ \frac{\alpha-\alpha^p}{1-\alpha} + \alpha^p } 
= \frac{p \alpha \mu(\alpha)^{p-1} (1-\alpha) }{ \alpha-\alpha^{p+1} } \\
&= \frac{p \mu(\alpha)^{p-1} (1-\alpha) }{ 1-\alpha^p } 
= \alpha^{1-1/p} \frac{ g(\alpha)^{1-1/p} }{ h(\alpha) },
\end{align*}
where
\begin{align*}
g(\alpha) &= \frac{1-\alpha^{p-1}}{(p-1)(1-\alpha)} = \frac{1}{p-1} \sum_{j=0}^{p-2} \alpha^j, \\
h(\alpha) &= \frac{1-\alpha^p}{p(1-\alpha)} = \frac{1}{p} \sum_{j=0}^{p-1} \alpha^j.
\end{align*}
Since $0 < h(\alpha) < g(\alpha) < 1$ for every $\alpha \in (0,1)$, it follows that
\[
\frac{g(\alpha)^{1-1/p}}{h(\alpha)} > \frac{g(\alpha)}{h(\alpha)} > 1.
\]
\end{proof}

Lemma~\ref{lemma:lowerbound} implies
\begin{equation}  \label{eq:alphait0new}
\alpha_{k+1} \ge \alpha_k^{1-1/p}
\end{equation}
for every $k$, so 
\[
\alpha_k \ge \alpha^{(1-1/p)^k}.
\]
Thus, we will have $\alpha_{k_1} \ge 1/e$ if $ \alpha^{(1-1/p)^{k_1}} \ge 1/e$, which means
\begin{equation}  \label{eq:k1def}
k_1 \geq  \frac{\log \log \frac{1}{\alpha}}{\log (\frac{p}{p-1})}.
\end{equation}

\paragraph{Stage 2}
As mentioned previously, $k_2$ is a constant independent of $\alpha$ and $\epsilon$. 

\paragraph{Stage 3}
We now determine $k_3$ such that $\frac{1-\alpha_{k_1+k_2+k_3}}{1+\alpha_{k_1+k_2+k_3}} \le \epsilon$.  By Lemma~\ref{lemma:alphastar},
\[
\frac{1 - \alpha_{k+1}}{1+\alpha_{k+1}}
= \frac{1-H(\alpha_k)}{1+H(\alpha_k)} 
\le \frac{p}{2} \left( \frac{1-\alpha_k}{1+\alpha_k} \right)^2 \\
\]
for $k \ge k_1+k_2$.
In terms of $\delta_k := \frac{p}{2}\left( \frac{1-\alpha_k}{1+\alpha_k} \right)$, we have $\delta_{k+1} \le \delta_k^2$, so
\[
\delta_{k_1+k_2+k} \le \delta_{k_1+k_2}^{2^k} \le \left( \frac{p}{2} \left( \frac{1-\alpha^*}{1+\alpha^*} \right) \right)^{2^k}.
\]
By~(\ref{eq:wlog}), $\frac{p}{2} \left( \frac{1-\alpha^*}{1+\alpha^*} \right) < 1$, so we will have $\frac{1-\alpha_{k_1+k_2+k_3}}{1+\alpha_{k_1+k_2+k_3}} \le \epsilon$ if 
\[
\left( \frac{p}{2} \left( \frac{1-\alpha^*}{1+\alpha^*} \right) \right)^{2^{k_3}} \le \frac{p}{2} \epsilon,
\]
i.e.
\[
k_3 \ge \frac{ \log\log\frac{2}{\epsilon p}  - \log\log \frac{2}{p}\left( \frac{1+\alpha^*}{1-\alpha^*} \right) }{ \log 2 }.
\]

Finally, by taking $\alpha=\epsilon/2$ we ensure that the error on $[0,\alpha^p]$ is bounded by $\epsilon$ (recall~\eqref{combinedbound}), so the error on $[0,1]$ is bounded by $\epsilon$. 

We illustrate the process in Figure~\ref{fig:err}, 
where we fix integers\footnote{$p=31$ is a somewhat arbitrary prime number, chosen in view of the number of days per month.} $p$ and $k$, and numerically find the value of $\alpha\in (0,1)$ and accordingly $\epsilon=\frac{1-\alpha_k}{1+\alpha_k}=2\alpha$ such that with the $(k,p,p-1)$-composite rational approximant $\widetilde f_k$ the error is $\max_{x\in[\alpha^p,1]}|x^{1/p}-\widetilde f_k(x)|\leq\epsilon$, achieved at $x=1$, and the error on $[0,\alpha^p]$ is bounded by $\epsilon$. 
Observe that the maximum errors on $[0,\alpha^p]$ and $[\alpha^p,1]$ are 
 not equal but of the same order, suggesting the near optimality of our composite rational approximants.
\begin{figure}[htbp]
  \centering
  \begin{minipage}[t]{0.495\hsize}
      \includegraphics[width=0.9\textwidth]{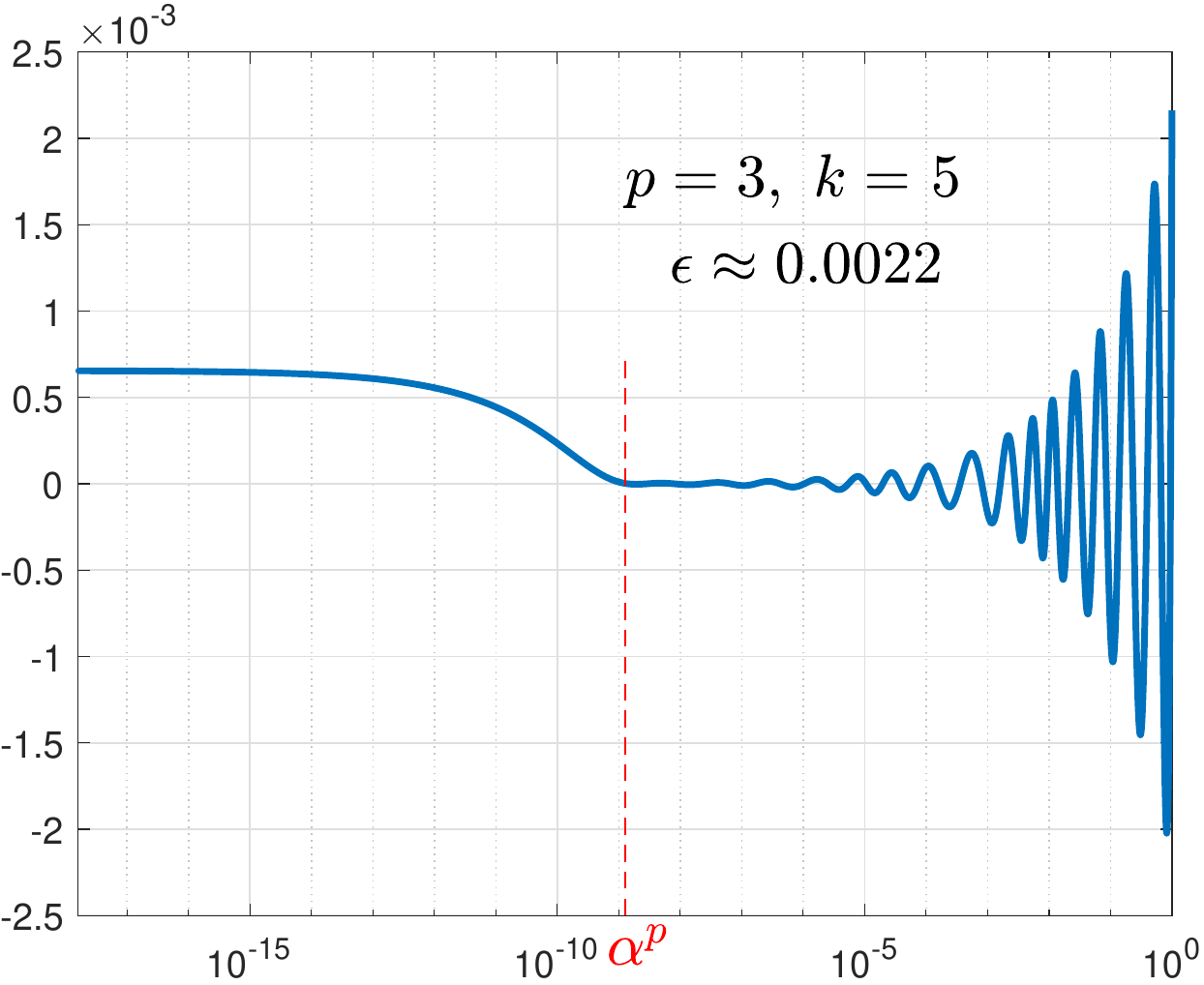}
  \end{minipage}   
  \begin{minipage}[t]{0.495\hsize}
      \includegraphics[width=0.9\textwidth]{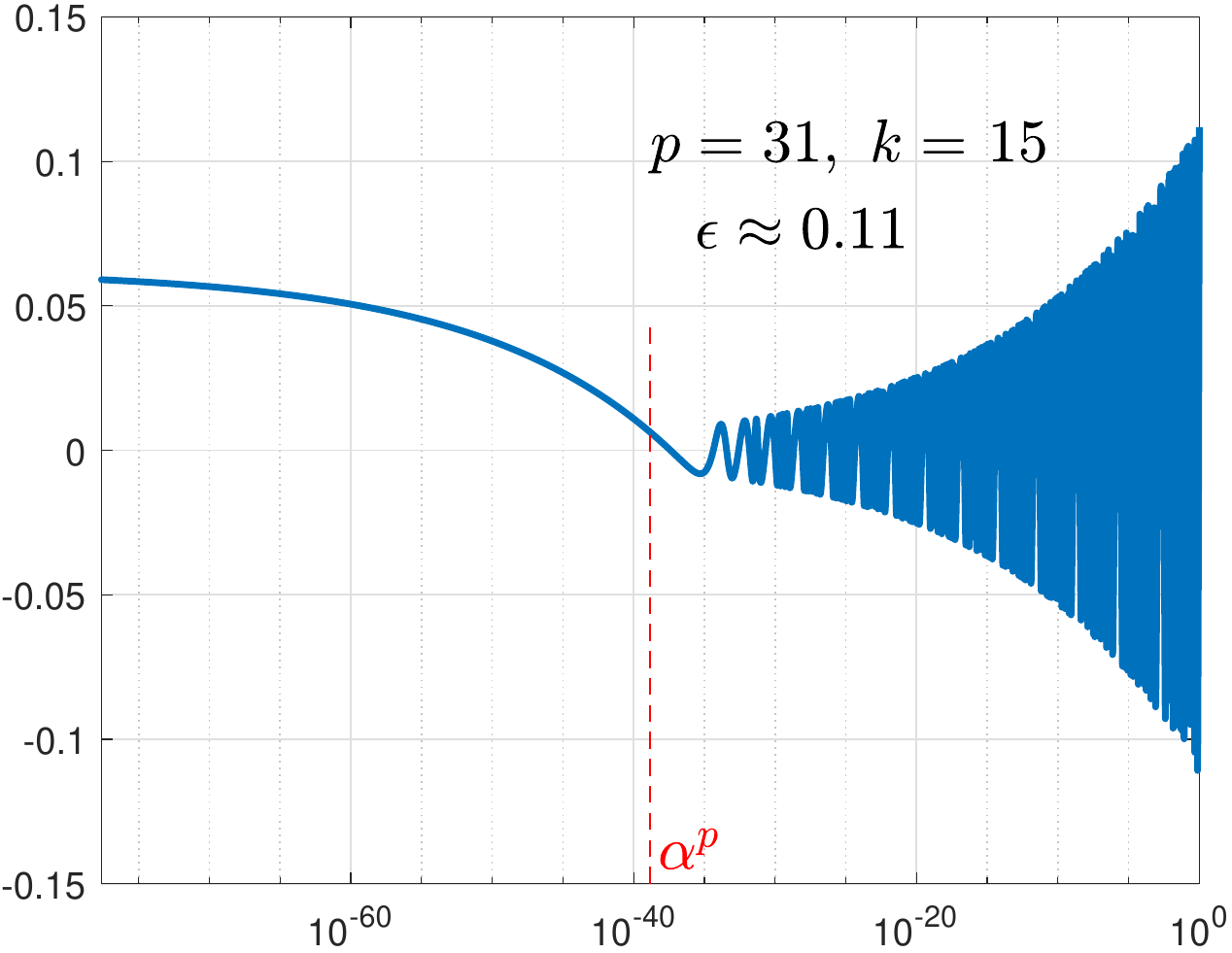}
  \end{minipage}
  \caption{Error curves $\widetilde{f}_k(x)-x^{1/p}$. 
Note that the error on $[0,\alpha^p]$ is bounded by that on $[\alpha^p,1]$, which is $\epsilon=2\alpha$ in both cases.}
  \label{fig:err}
\end{figure}

Putting these inequalities together, we conclude that 
\begin{equation}  \label{eq:kval}
k = \frac{\log \log \frac{2}{\epsilon}}{\log (\frac{p}{p-1})} + \widetilde{k}_2 + \frac{ \log\log\frac{2}{\epsilon p} }{\log 2}
\end{equation}
recursions are enough to yield accuracy $\epsilon$, where $\widetilde{k}_2$ is an integer satisfying
\[
k_2 - \frac{\log\log \frac{2}{p} \frac{1+\alpha^*}{1-\alpha^*}}{\log 2} \le \widetilde{k}_2 \le k_2 - \frac{\log\log \frac{2}{p} \frac{1+\alpha^*}{1-\alpha^*}}{\log 2} + 1.
\]
Since $k$ recursions translate into a rational function $\widetilde{f}_k$ of type $(p^{k-1},p^{k-1}-1)$, it follows that the degree $n$ of the rational function $\widetilde{f}_k$ achieving accuracy $\epsilon$ is 
\[
n = p^{\frac{\log \log \frac{2}{\epsilon}}{\log (\frac{p}{p-1})}+
\widetilde{k}_2+\frac{\log\log\frac{2}{\epsilon p}}{\log 2}-1  }. 
\]
We rewrite this to express the error with respect to the degree $n$. Taking the logarithm and absorbing the constant $-1$ into $\widetilde{k}_2$, we get 
\begin{equation}\label{eq:workoutk}
\log n = \left(\frac{\log \log \frac{2}{\epsilon}}{\log (\frac{p}{p-1})}+
\widetilde{k}_2+\frac{\log\log\frac{2}{\epsilon p}}{\log 2} \right) \log p \le \left(\frac{\log \log \frac{2}{\epsilon}}{\log (\frac{p}{p-1})}+
\widetilde{k}_2+\frac{\log\log\frac{2}{\epsilon}}{\log 2} \right) \log p.
\end{equation}
Hence,
\[
\log\log\frac{2}{\epsilon} \ge \frac{\log n-\widetilde{k}_2\log p}{
\log p\large(\frac{1}{\log (\frac{p}{p-1})}+
\frac{1}{\log 2}\large) }. 
\]
Thus, defining 
\begin{equation}  \label{eq:cp}
c:= \frac{1}{\log p\large(\frac{1}{\log (\frac{p}{p-1})}+
\frac{1}{\log 2}\large) }= 
\frac{ \log 2\log \frac{p}{p-1}}{\log p \log \frac{2p}{p-1}  }, 
\end{equation}
 we have 
\[
\log\frac{2}{\epsilon} \ge \left(\frac{n}{p^{\widetilde{k}_2}}\right)^{c},
\]
and therefore, writing $\widetilde{b} = 1/p^{c\widetilde{k}_2}$, we arrive at 
\[
\epsilon \le 2\exp(-\widetilde{b} n^{c}). 
\]
This bound holds when $n$ is a sufficiently large power of $p$.  To handle the case in which $n \in \mathbb{N}$ is not a power of $p$, we note that $\lfloor \log_p n \rfloor + 1$ recursions yield a rational function of type $(p^{\lfloor \log_p n \rfloor}, p^{\lfloor \log_p n \rfloor}-1)$, and for $n$ large enough ($n \ge N$, say) this function has error bounded above by
\[
2\exp(-\widetilde{b} (p^{\lfloor \log_p n \rfloor})^c) \le  2\exp(-\widetilde{b} p^{-c} n^c)  \le \exp(-(\widetilde{b} p^{-c}-N^{-c}\log 2) n^c). 
\]
Taking $N$ large enough yields Theorem~\ref{thm:main} with $b = \widetilde{b} p^{-c}-N^{-c}\log 2>0$.

It is easy to see by comparing~(\ref{eq:error1sector}) with~(\ref{combinedbound}) that the same analysis, this time choosing $\alpha=\epsilon$ rather than $\alpha=\epsilon/2$, also yields~\eqref{eq:sectthm} in Theorem~\ref{thm:sector}.

It remains to establish~\eqref{eq:sectthmhat}. 
For this, we take $\alpha$ fixed and use a similar argument. In this case $k_1,k_2$ can both be regarded as constants independent of $\epsilon$, since 
the error in the interval $[0,\alpha^p]$ is irrelevant. 
Therefore we write $\widehat k:=k_1+\widetilde{k}_2$, and in place of~\eqref{eq:workoutk}, the lowest degree $n$ required for $\epsilon$ accuracy on $S_{p,\alpha}$ satisfies 
$\log n\leq \left(\widehat k+\frac{\log\log\frac{2}{\epsilon p}}{\log 2} \right) \log p \le \left(\widehat k+\frac{\log\log\frac{1}{\epsilon}}{\log 2} \right) \log p$. Thus 
 defining 
\begin{equation}  \label{eq:cphat}
\widehat c:= \frac{\log 2}{\log p} (>c), 
\end{equation}
 we have 
$\log\frac{1}{\epsilon} \ge \big(\frac{n}{p^{\widehat{k}}}\big)^{\widehat c},$
and so setting $\widehat{b} = 1/p^{\widehat c\widehat{k}}$ we obtain 
$\epsilon \le \exp(-\widehat{b} n^{\widehat c})$, as required. 
\hfill$\square$

Note that for $\alpha\ll 1$ we have $\widehat k\approx k_1=
 \frac{\log \log\frac{1}{\alpha}}{\log (\frac{p}{p-1})}$, so $\widehat b\approx (\log\frac{1}{\alpha})^{-\widehat C}$  for some $\widehat C>0$ depending only on $p$, so $\widehat b$ scales like an inverse power of $\log\frac{1}{\alpha}$.

\section{Examples} \label{sec:examples}
In Figure~\ref{fig:err2} we illustrate our main result~\eqref{eq:mainthm} on approximation of $x^{1/p}$. For integers $k=1,2,\ldots,$, we compute the error $\epsilon$ of the composite rational approximants as in Figure~\ref{fig:err}, 
and plot the errors against $p^{ck}(\approx n^c)$ for $p\in\{2,5,31\}$ in log-scale. The plots also show least-squares affine fits to the convergence data for each $p$. The fact that the affine fits closely trace the data suggests the exponent $c$ in \eqref{eq:cp} 
is sharp, especially for small values of $p$. For the $p=31$ plot, which ends early because computing further data was infeasible (note e.g.\; that $\alpha^p< 10^{-70}$ for $k\geq 15$), there is a slight bend in the convergence, which suggests that our $c$ in~\eqref{eq:exponentc} might be a slight underestimate for large $p$.
\begin{figure}[htbp]
  \centering
\includegraphics[width = 80mm]{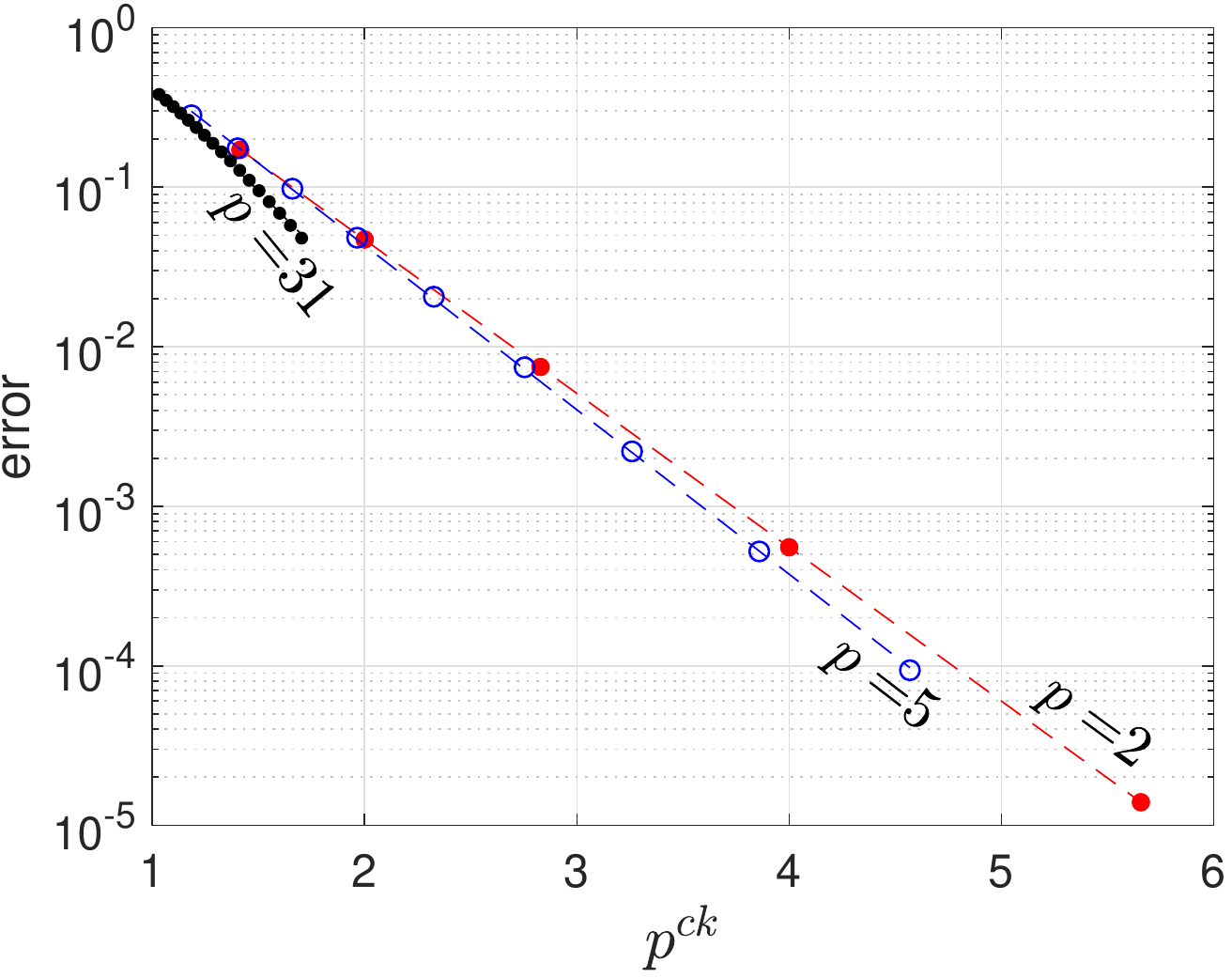}
  \caption{Error history 
$\max_{x \in [0,1]} | \widetilde{f}_k(x) - x^{1/p}  |$
for varying $k$ for $p\in\{2,5,31\}$, along with linear fits shown as dashed lines. 
}
  \label{fig:err2}
\end{figure}

Finally, Figure~\ref{fig:sector} shows the error of the approximant $\widetilde{g}_k(z)$ to $\operatorname{sect}_p(z)$, which clearly exhibits equioscillation. Note how increasing $k$ results in progressively smaller error (in log-scale), reflecting the double-exponential convergence. 
The error curves $|\widetilde{g}_k(z) - \operatorname{sect}_p(z)|$ look identical on each of the segments $[\alpha,1]\exp(2\pi{\rm i}j/p)$ for $j=0,\ldots,p-1$.

\begin{figure}[htpb]
  \begin{minipage}[t]{0.5\hsize}
      \includegraphics[width=0.9\textwidth]{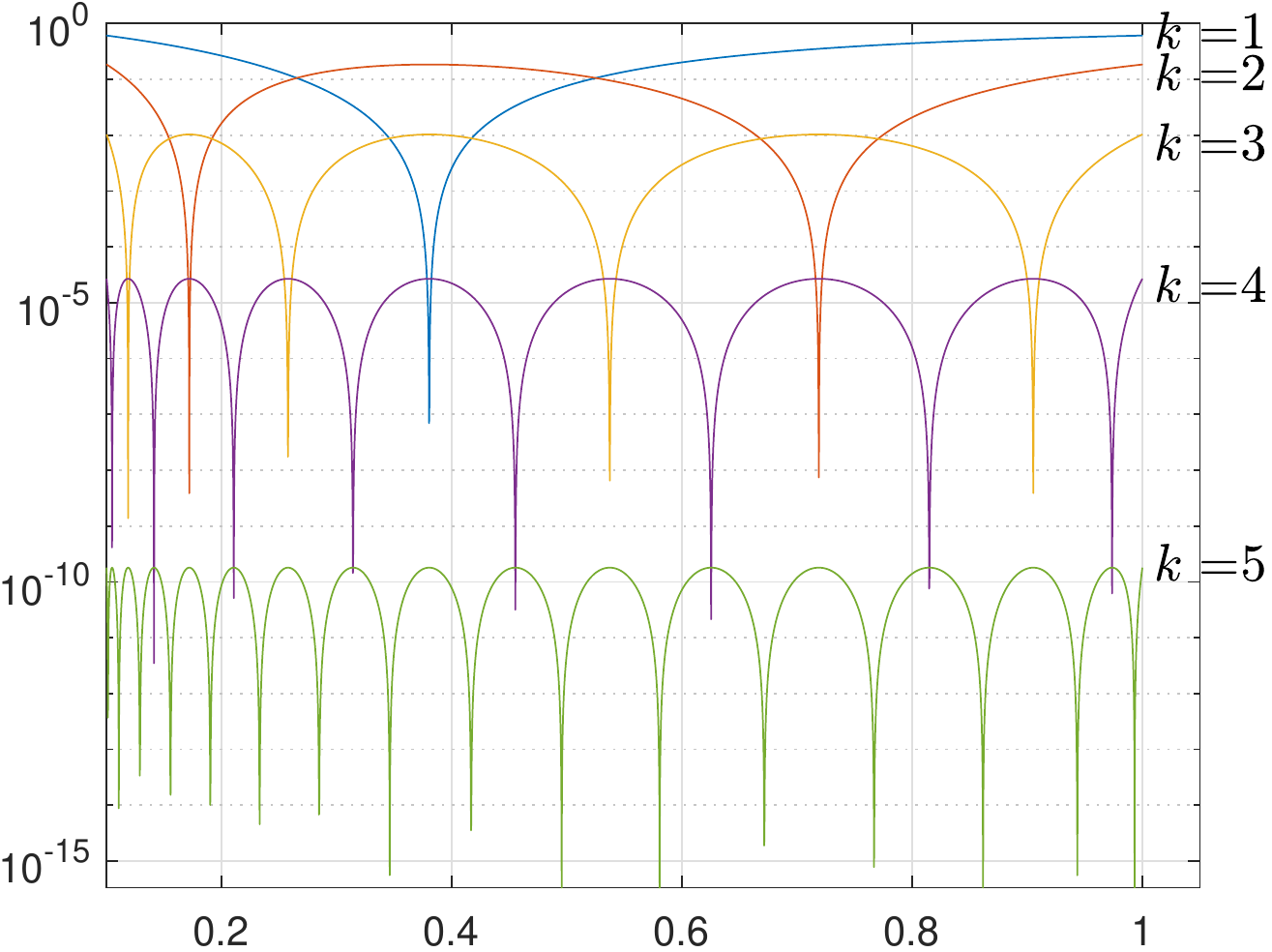}
  \end{minipage}   
  \begin{minipage}[t]{0.5\hsize}
      \includegraphics[width=0.9\textwidth]{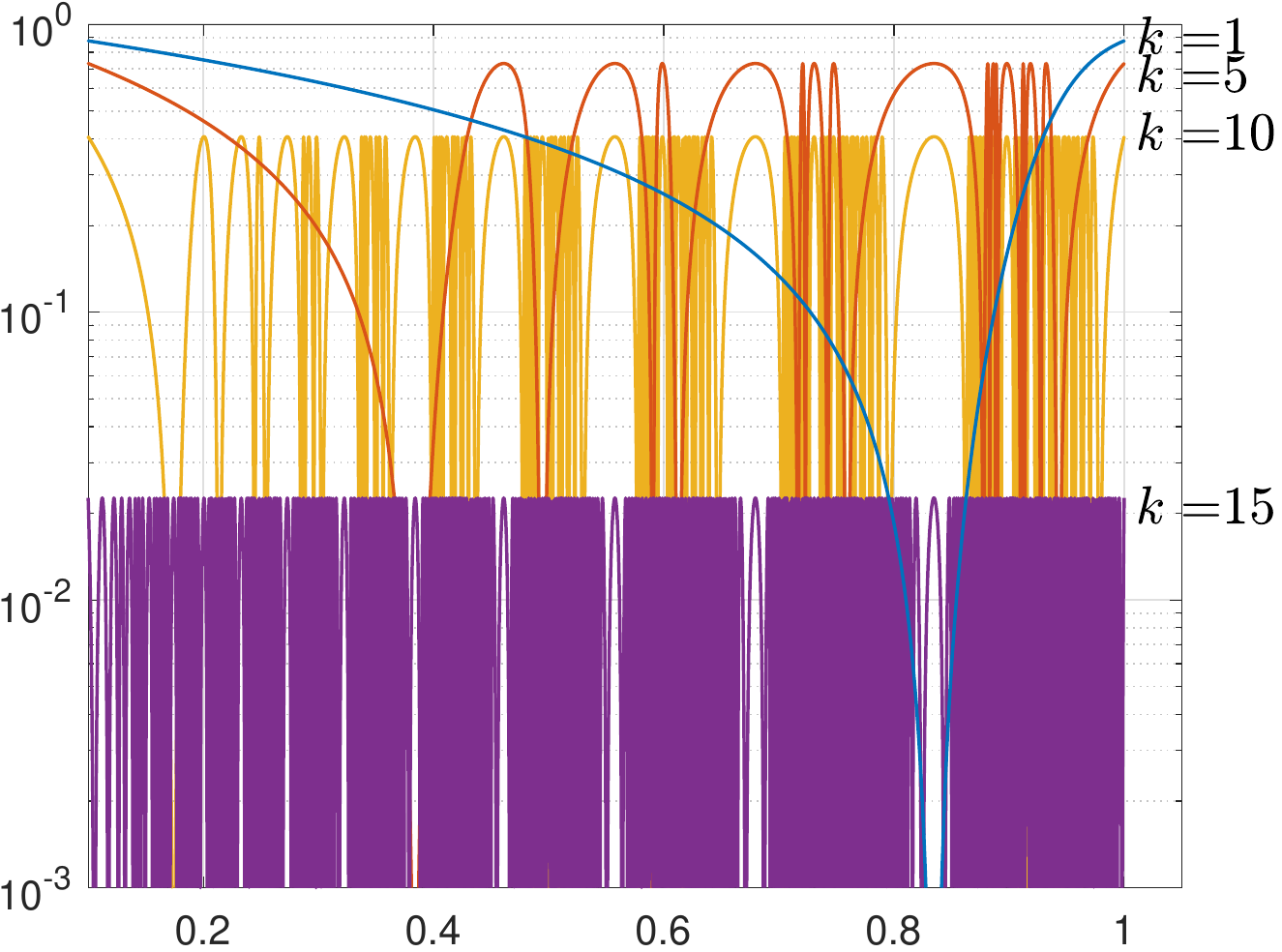}
  \end{minipage}
  \caption{
    Error $|\widetilde{g}_k(z) - \operatorname{sect}_p(z)|$ on $[\alpha,1]$ for $\alpha=0.1$, $p=3$ (left) and $p=31$ (right). The fact that the plots do not appear to go down to 0 between equioscillation points is simply an artifact of the plotting scheme, which is based on $10^4$ equispaced sample points.
}
  \label{fig:sector}
\end{figure}

\subsection*{Acknowledgment}
We thank Alex Townsend, a discussion with whom inspired this work.

\bibliographystyle{abbrv}
\bibliography{bib2}

\def\noopsort#1{}\def\l{\char32l}\def\v#1{{\accent20 #1}}
  \let\^^_=\v\def\hbk{hardback}\def\pbk{paperback}
\begin{thebibliography}{10}

\bibitem{elementselliptic}
N.~I. Akhiezer.
\newblock {\em Elements of the Theory of Elliptic Functions}, volume~79 of {\em
  Translations of Mathematical Monographs}.
\newblock American Mathematical Society, 1990.

\bibitem{Beckermannmatfun}
B.~Beckermann.
\newblock Optimally scaled {N}ewton iterations for the matrix square root.
\newblock {\em FUN13: Advances in Matrix Functions and Matrix Equations
  workshop}, 2013.

\bibitem{beckermann2017singular}
B.~Beckermann and A.~Townsend.
\newblock On the singular values of matrices with displacement structure.
\newblock {\em SIAM J. Matrix Anal. Appl.}, 38(4):1227--1248, 2017.

\bibitem{braess1986nonlinear}
D.~Braess.
\newblock {\em Nonlinear Approximation Theory}.
\newblock Springer, 1986.

\bibitem{gawlik2018pth}
E.~S. Gawlik.
\newblock Rational minimax iterations for computing the matrix pth root.
\newblock {\em arXiv preprint arXiv:1903.06268}, 2019.

\bibitem{gawlik2018zolotarev}
E.~S. Gawlik.
\newblock Zolotarev iterations for the matrix square root.
\newblock {\em SIAM J. Matrix Anal. Appl.}, 40(2):696--719, 2019.

\bibitem{gonvcar1967rapidity}
A.~Gon{\v{c}}ar.
\newblock On the rapidity of rational approximation of continuous functions
  with characteristic singularities.
\newblock {\em Mathematics of the USSR-Sbornik}, 2(4):561, 1967.

\bibitem{Higham2008FM}
N.~J. Higham.
\newblock {\em Functions of Matrices: {Theory} and Computation}.
\newblock SIAM, Philadelphia, PA, USA, 2008.

\bibitem{king1971improved}
R.~F. King.
\newblock Improved {N}ewton iteration for integral roots.
\newblock {\em Mathematics of Computation}, 25(114):299--304, 1971.

\bibitem{lecun2015deep}
Y.~LeCun, Y.~Bengio, and G.~Hinton.
\newblock Deep learning.
\newblock {\em Nature}, 521(7553):436, 2015.

\bibitem{meinardus1980optimal}
G.~Meinardus and G.~Taylor.
\newblock Optimal partitioning of {Newton's} method for calculating roots.
\newblock {\em Mathematics of Computation}, 35(152):1221--1230, 1980.

\bibitem{mysirev}
Y.~Nakatsukasa and R.~W. Freund.
\newblock Computing fundamental matrix decompositions accurately via the matrix
  sign function in two iterations: The power of {Z}olotarev's functions.
\newblock {\em SIAM Rev.}, 58(3):461--493, 2016.

\bibitem{ninomiya1970best}
I.~Ninomiya.
\newblock Best rational starting approximations and improved {N}ewton iteration
  for the square root.
\newblock {\em Math. Comp.}, 24(110):391--404, 1970.

\bibitem{petrushev2011rational}
P.~P. Petrushev and V.~A. Popov.
\newblock {\em Rational Approximation of Real Functions}.
\newblock Cambridge University Press, 2011.

\bibitem{rutishauser1963betrachtungen}
H.~Rutishauser.
\newblock Betrachtungen zur {Q}uadratwurzeliteration.
\newblock {\em Monatshefte f{\"u}r Mathematik}, 67(5):452--464, 1963.

\bibitem{stahl2003best}
H.~R. Stahl.
\newblock Best uniform rational approximation of $x^\alpha$ on [0, 1].
\newblock {\em Acta Math.}, 190(2):241--306, 2003.

\bibitem{trefethenatap}
L.~N. Trefethen.
\newblock {\em Approximation Theory and Approximation Practice}.
\newblock SIAM, Philadelphia, 2013.

\bibitem{wachspress1962}
E.~Wachspress.
\newblock Positive definite square root of a positive definite square matrix.
\newblock {\em Unpublished}, 1962.

\end{thebibliography}

\end{document}